\documentclass[12pt, leqno]{amsart}
\usepackage{amssymb}
\usepackage{amsthm}
\usepackage{amsmath}
\usepackage{amsfonts}
\usepackage[mathscr]{eucal}
\usepackage{enumerate}

\theoremstyle{plain}
\newtheorem{Thm}{Theorem}[section]
\newtheorem{Cor}{Corollary}[section]
\newtheorem{Lem}{Lemma}[section]

\newtheorem{Rem}{Remark}[section]

\theoremstyle{definition}
\newtheorem{Def}{Definition}[section]
\theoremstyle{remark}

\def\R{\mathbb{R}}

\def\N{\mathbb{N}}

\begin{document}
\baselineskip=17pt

\title[Stability type results]
{Stability type results concerning the fundamental equation of information of multiplicative type}
\author {Eszter Gselmann}
\address{
Institute of Mathematics\\
University of Debrecen\\
H-4010 Debrecen\\
P. O. Box 12 \\
Hungary}
\email{gselmann@math.klte.hu}
\keywords{Stability, fundamental equation of information of multiplicative type}
\thanks{This research has been supported by the Hungarian Scientific Research Fund (OTKA) Grant NK 68040. }
\subjclass[1991]{39B82}

\begin{abstract}
The paper deals with the stability of the fundamental equation of information of multiplicative type.
It will be proved that the equation in question is stable in the sense of Hyers and Ulam
under some assumptions. This result will be applied to
prove the stability of a system of functional equations that characterizes the recursive
measures of information of multiplicative type.

\end{abstract}

\maketitle

\section{Introduction}
The stability theory of functional equations deals with the following question: When it is true that the solution of an equation differing slightly from a given one, must of necessity be close to the solution of the given equation. In case of a positive answer to the previous problem, we say that the equation in question is \textit{stable}.
This problem was raised by Ulam (see \cite{Ula40}) and answered by Hyers who proved that the Cauchy equation is stable (\cite{Hye41}). Since then, this result has been extended and generalized in several ways (see e.g. \cite{For95}, \cite{Ger94} and \cite{HIR98}).
The investigation of the stability of the exponential Cauchy equation highlighted a new phenomenon which is now usually called \textit{superstability} (see e.g. \cite{HIR98}). The question of superstability is also dealt with in this paper.
Solving a stability problem, raised in \cite{Mak07} we give an affirmative answer to the case of higher dimensional information functions. \\
Throughout this paper let $k$ and $n$ be an arbitrary but fixed positive integers and denote
	\[
	\Gamma_{n}:=\left\{\left(p_{1}, \ldots, p_{n}\right)\in\R^{kn}\vert p_{i}\geq \mathbf{0}, \sum^{n}_{i=1}p_{i}=\mathbf{1}\right\}
\]
and
	\[
	D_{k}:=\left\{\left(x, y\right)\in\R^{2k}\vert x, y\in[0, 1[^{k}, x+y\leq \mathbf{1}\right\}.
\]
Furthermore, let
\[
\R^{k+}=\left\{x\in \R^{k}\vert x>\mathbf{0}\right\}.
\]
Here $\mathbf{1}$ represents the $k$-vector $\left(1, \ldots, 1\right)\in\R^{k}$ and all operations on vectors are to be done componentwise, i.e., $p_{i}\geq\mathbf{0}$ denotes that all coordinates of the vector $p_{i}\in\R^{k}$ are non-negative and we write $x+y\leq\mathbf{1}$ if $x_{i}+y_{i}\leq 1$ holds for all $i=1, \ldots, k$, where $x_{i}$ and $y_{i}$ denote the $i^{th}$ coordinates of the vector $x$ and $y$, respectively. \\
In what follows, we present some basic results from the theory of functional equations which we shall use throughout the paper, these results can be found for instance in \cite{Kuc85}. \\
\noindent
A function $M:[0, 1]^{k}\rightarrow\R$ is called \emph{multiplicative}, if
\begin{equation}
	M\left(x\cdot y\right)=M\left(x\right)\cdot M\left(y\right)
\end{equation}
holds for all $x, y\in\left[0, 1\right]^{k}$.\\
\noindent
We say that $A:\left[0, 1\right]^{k}\rightarrow\R$ is \emph{additive on $D_{k}$} if
\begin{equation}
A\left(x+y\right)=A\left(x\right)+A\left(y\right)
\end{equation}
holds for all pairs $\left(x, y\right)\in D_{k}$.

\begin{Lem}
If $M:[0,1]^{k}\rightarrow\R$ is both multiplicative on $[0,1]^{k}$ and additive on $D_{k}$, then $M$ is either identically zero or a projection, i.e.,
\begin{equation}
	M\left(x\right)=M\left(x_{1}, \ldots, x_{k}\right)=x_{j}, \qquad x\in [0,1]^{k},
\end{equation}
for some $j\in \left\{0, \ldots, k\right\}$.
\end{Lem}
In the proof of our theorem we shall use the following lemma.
\begin{Lem}\label{lemma}
Let $M:[0,1]^{k}\rightarrow\R$ be a multiplicative function. Then the following statements are equivalent.
\begin{enumerate}[(i)]
	\item $M$ is additive on $D_{k}$;
	\item $M\left(x\right)+M\left(\mathbf{1}-x\right)=1$ holds for all $x\in [0,1]^{k}$.
\end{enumerate}
\end{Lem}

\begin{Lem}
Let $M:[0,1]^{k}\rightarrow\R$ be a multiplicative function then
\[
M\left(x\right)\geq 0
\]
holds for all $x\in [0,1]^{k}$.
\end{Lem}

\begin{Lem}
Let $M:[0,1]^{k}\rightarrow\R$ be multiplicative. Then
\begin{equation}
M\left(x\right)=M\left(x_{1}, \ldots, x_{k}\right)=\prod^{k}_{i=1}m_{i}\left(x_{i}\right)
\end{equation}
for all $x\in [0,1]^{k}$, where each $m_{i}:[0,1]\rightarrow\R$ is multiplicative $(i=1, \ldots, k)$.
\end{Lem}
Now we fall to dealing with information measures (see \cite{AD75}, \cite{ESS98}).
\begin{Def}
A sequence of functions $I_{n}:\Gamma_{n}\rightarrow\R$ ($n=2, 3, \ldots$) is called \emph{information measure}.
\end{Def}
The usual information-theoretical interpretation is that $I_n\left(p_{1},\ldots, p_{n}\right)$
is a measure of uncertainty as to the outcome of an experiment having $n$
possible outcomes with probabilities $p_{1},\ldots, p_{n}$. \\
Some desiderata for measures of information can be found in \cite{AD75} as well as in \cite{ESS98}. Although in this paper we will use only the following properties.
\begin{Def}
The sequence of functions $I_{n}:\Gamma_{n}\rightarrow\R$ ($n=2, 3, \ldots$) is
\begin{enumerate}[(i)]
	\item \emph{M-recursive}, if
	\[
	I_{n}\left(p_{1}, \ldots, p_{n}\right)=I_{n-1}\left(p_{1}+p_{2}, p_{3}, \ldots, p_{n}\right)+
	M\left(p_{1}+p_{2}\right)I_{2}\left(\frac{p_{1}}{p_{1}+p_{2}}, \frac{p_{2}}{p_{1}+p_{2}}\right)
\]
holds for all $n=3, 4, \ldots$ and $\left(p_{1}, \ldots, p_{n}\right)\in\Gamma_{n}$, with some multiplicative function
$M:\left[0, 1\right]^{k}\rightarrow\R$ and with the convention $\frac{0}{0+0}=0$.
\item \emph{3-semisymmetric}, if
	\[
	I_{3}\left(p_{1}, p_{2}, p_{3}\right)=I_{3}\left(p_{1}, p_{3}, p_{2}\right)
\]
holds for all $\left(p_{1}, p_{2}, p_{3}\right)\in\Gamma_{3}$.
\end{enumerate}
\end{Def}
\noindent
Using the following theorem the characterization of information measures can be transformed into solving functional equations, see e.g., \cite{ESS98}.
\begin{Thm}
If the sequence of functions $I_{n}:\Gamma_{n}\rightarrow\R$ $\left(n=2, 3, \ldots\right)$ is M-recursive and 3-semisymmetric then the function
$f:[0, 1]^{k}\rightarrow\R$ defined by
	\[
	f(x):=I_{2}\left(\mathbf{1}-x, x\right)
\]
satisfies the so-called fundamental equation of information of multiplicative type M, i.e.,
\begin{equation}\label{FEIM}
f(x)+M\left(\mathbf{1}-x\right)f\left(\frac{y}{\mathbf{1}-x}\right)=
f(y)+M\left(\mathbf{1}-y\right)f\left(\frac{x}{\mathbf{1}-y}\right)
\end{equation}
for all $\left(x, y\right)\in D_{k}$.
\end{Thm}

\section{Known results}

In \cite{Mak07} it is proved that (\ref{FEIM}) is stable, moreover superstable assuming that $k=1$ and the function $M:[0, 1]^{k}\rightarrow\R$ is the power function, i.e.,  the stability of the following equation was investigated
	\[
	f(x)+\left(1-x\right)^{\alpha}f\left(\frac{y}{1-x}\right)=
	f\left(y\right)+\left(1-y\right)^{\alpha}f\left(\frac{x}{1-y}\right),
\]
where $0<\alpha \neq 1$. \\
In \cite{Mor01} a stability type result is proved for $k=1$ and $\alpha=1$, i.e., for the Shannon entropy.
However, Morando's theorem claims stability only on the rationals.

\section{Main result}

In this section we will show stability type
results concerning the fundamental equation of
information of multiplicative type. Our main result is the following
\begin{Thm}\label{main}
Let $\varepsilon\geq 0$ be arbitrary, $M:[0,1]^{k}\rightarrow\R$ be multiplicative but not additive and $f:[0,1]^{k}\rightarrow\R$ be a function. Assume that
\begin{equation}\label{stab}
\left|f\left(x\right)+M\left(\mathbf{1}-x\right)f\left(\frac{y}{\mathbf{1}-x}\right)-
f\left(y\right)-M\left(\mathbf{1}-y\right)f\left(\frac{x}{\mathbf{1}-y}\right)\right|\leq \varepsilon
\end{equation}
holds for all $\left(x, y\right)\in D_{k}$. Then there exist $a, b\in\R$ such that
\begin{equation}\label{stab.ineq}
\begin{array}{l}
\left|f\left(x\right)-\left(aM\left(x\right)+b\left(M\left(\mathbf{1}-x\right)-1\right)\right)\right|\leq  \\
\left|M\left(q*\right)+M\left(\mathbf{1}-q*\right)-1\right|^{-1} \cdot
\left(4\varepsilon +3\varepsilon M\left(\mathbf{1}-xq*\right)\right)
\end{array}
\end{equation}
holds for all $x\in [0, 1]^{k}$, where $q*\in ]0, 1[^{k}$ is such that
$M\left(q*\right)+M\left(\mathbf{1}-q*\right)\neq 1$.
\end{Thm}
\begin{proof}
Define the function $F$ on $]0,1[^{k}\times [0,1]^{k}$ by
\begin{equation}
F\left(p, q\right)= f\left(\mathbf{1}-p\right)+M\left(p\right)f\left(q\right)-
f\left(pq\right)-M\left(\mathbf{1}-pq\right)f\left(\frac{\mathbf{1}-p}{\mathbf{1}-pq}\right).
\end{equation}
Then equation (\ref{stab}) with the substitution
$x=\mathbf{1}-p$ and $y=pq$ implies that
\begin{equation}\label{eps}
\left|F\left(p, q\right)\right|\leq \varepsilon
\end{equation}
holds for all $p, q\in ]0, 1[^{k}$.
On the other hand we have that
\begin{equation}
\begin{array}{l}
\left|\left[M\left(q\right)+M\left(\mathbf{1}-q\right)-1\right]\cdot
\left[f\left(p\right)-f\left(\mathbf{1}\right)M\left(p\right)\right]\right. \\
-\left[M\left(p\right)+M\left(\mathbf{1}-p\right)-1\right]\cdot
\left[f\left(q\right)-f\left(\mathbf{1}\right)M\left(q\right)\right]\left.\right| \\
=F\left(q, p\right)+F\left(p, q\right)-F\left(q, \mathbf{1}\right)+F\left(p, \mathbf{1}\right) \\
+M\left(\mathbf{1}-pq\right)\left[F\left(\frac{\mathbf{1}-p}{\mathbf{1}-pq}, \mathbf{1}\right)
+F\left(\frac{\mathbf{1}-p}{\mathbf{1}-pq}, \mathbf{1}\right)
-F\left(\frac{\mathbf{1}-p}{\mathbf{1}-pq}, q\right)\right]
\end{array}
\end{equation}
holds for all $p, q \in ]0,1[^{k}$. \\
Now using equation (\ref{eps}) we get that
\begin{equation}\label{becsl}
\begin{array}{l}
\left|\left[M\left(q\right)+M\left(\mathbf{1}-q\right)-1\right]\cdot
\left[f\left(p\right)-f\left(\mathbf{1}\right)M\left(p\right)\right]\right. \\
-\left[M\left(p\right)+M\left(\mathbf{1}-p\right)-1\right]\cdot
\left[f\left(q\right)-f\left(\mathbf{1}\right)M\left(q\right)\right]\left.\right| \\
\leq 4\varepsilon +3\varepsilon M\left(\mathbf{1}-pq\right)
\end{array}
\end{equation}
Since $M$ is not additive there exists a $q*\in ]0,1[^{k}$ such that
\begin{equation}\label{neq}
M\left(q*\right)+M\left(\mathbf{1}-q*\right)\neq 1.
\end{equation}
\noindent
Then with the substitution $q=q*$ in (\ref{becsl}) we have that
\begin{equation}
\begin{array}{l}
\left|\left[M\left(q*\right)+M\left(\mathbf{1}-q*\right)-1\right]\cdot
\left[f\left(p\right)-f\left(\mathbf{1}\right)M\left(p\right)\right]\right. \\
-\left[M\left(p\right)+M\left(\mathbf{1}-p\right)-1\right]\cdot
\left[f\left(q*\right)-f\left(\mathbf{1}\right)M\left(q*\right)\right]\left.\right| \\
\leq 4\varepsilon +3\varepsilon M\left(\mathbf{1}-pq*\right)
\end{array}
\end{equation}
Due to (\ref{neq}) we obtain that
\begin{equation}
\begin{array}{r}
\left|
\left[f\left(p\right)-f\left(\mathbf{1}\right)M\left(p\right)\right]-
\frac{f\left(q*\right)-f\left(\mathbf{1}\right)M\left(q*\right)}
{M\left(q*\right)+M\left(\mathbf{1}-q*\right)-1} \cdot
\left[M\left(p\right)+M\left(\mathbf{1}-p\right)-1\right]\right| \\
\leq \left|M\left(q*\right)+M\left(\mathbf{1}-q*\right)-1\right|^{-1} \cdot
\left(4\varepsilon +3\varepsilon M\left(\mathbf{1}-pq*\right)\right)
\end{array}
\end{equation}
Let
\[
a=f\left(\mathbf{1}\right)M\left(p\right)+
\frac{f\left(q*\right)-f\left(\mathbf{1}\right)M\left(q*\right)}
{M\left(q*\right)+M\left(\mathbf{1}-q*\right)-1}
\]
and
\[
b=\frac{f\left(q*\right)-f\left(\mathbf{1}\right)M\left(q*\right)}
{M\left(q*\right)+M\left(\mathbf{1}-q*\right)-1},
\]
\noindent
Therefore
\begin{equation}
\begin{array}{l}
\left|f\left(p\right)-\left[aM\left(p\right)+b\left(M\left(\mathbf{1}-p\right)-1\right)\right]\right| \leq \\
\left|M\left(q*\right)+M\left(\mathbf{1}-q*\right)-1\right|^{-1} \cdot
\left(4\varepsilon +3\varepsilon M\left(\mathbf{1}-pq*\right)\right)
\end{array}
\end{equation}
holds for all $p\in ]0,1[^{k}$. \\
\noindent
A direct calculation shows that (\ref{stab.ineq}) holds in case $p\in [0,1]^{k}\setminus ]0,1[^{k}$.
\end{proof}
\noindent
In what follows we define a function $K:[0,1]^{k}\rightarrow\R$ by
\begin{equation}\label{def.K}
K(x)=\frac{4\varepsilon+3\varepsilon M\left(\mathbf{1}-xq*\right)}
{\left|M\left(q*\right)+M\left(\mathbf{1}-q*\right)-1\right|}
\end{equation}
where $M:[0,1]^{k}\rightarrow\R$ is multiplicative but not additive,
$\varepsilon\geq 0$ arbitrary but fixed and $q*\in [0,1]^{k}$ is such that
$M\left(q*\right)+M\left(\mathbf{1}-q*\right)\neq 1$. \\
\noindent
Using the previous theorem we shall conclude the following.

\begin{Cor}
In case $\varepsilon=0$ in Theorem \ref{main}., then we get the general solution of equation
(\ref{FEIM}).
\end{Cor}

\begin{Cor}\label{rem1}
If the function $M:[0, 1]^{k}\rightarrow\R$ is bounded above by a constant $B\in\R$ then we get that inequality
(\ref{stab}) on $D_{k}$ implies
\begin{equation}
\begin{array}{l}
\left|f(x)-\left(aM\left(x\right)+b\left(M\left(\mathbf{1}-x\right)-1\right)\right)\right|\leq  \\
\left|M\left(q*\right)+M\left(\mathbf{1}-q*\right)-1\right|^{-1} \cdot
\left(4\varepsilon +3B\varepsilon\right)
\end{array}
\end{equation}
on $[0, 1]^{k}$, where $q*\in ]0, 1[^{k}$ is such that $M\left(q*\right)+M\left(\mathbf{1}-q*\right)\neq 1$.
\end{Cor}

\begin{Cor}
Due to Corollary \ref{rem1}. we obtain that the equation
\begin{equation}
f\left(x\right)+M\left(\mathbf{1}-x\right)f\left(\frac{y}{\mathbf{1}-x}\right)=
f\left(y\right)+M\left(\mathbf{1}-y\right)f\left(\frac{x}{\mathbf{1}-y}\right)
\end{equation}
is superstable on $D_{k}$ in case $M$ is bounded above.
\end{Cor}

\begin{Rem}
If $M\left(x\right)=x^{\alpha}$, $\left(x\in [0, 1]\right)$, where $0<\alpha\neq 1$
 then we get the result of Maksa (see \cite{Mak07}).
\end{Rem}
\noindent
Finally, the following theorem concerns the stability of a system of equations.

\begin{Thm}
Let $I_{n}$ be the sequence of functions
$I_{n}:\Gamma_{n}\rightarrow\R$ $(n\geq 2)$,
$M:[0,1]^{k}\rightarrow\R$ be a multiplicative function.
Suppose that there exist a sequence $\left(\varepsilon_{n}\right)$
of non-negative real numbers such that
\begin{equation}\label{ineq.syst}
\begin{array}{l}
\left| I_{n}\left(p_{1}, \ldots, p_{n} \right)-
I_{n-1}\left( p_{1}+p_{2}, p_{3}, \ldots, p_{n}\right) \right.\\
-M\left(p_{1}+p_{2}\right)
\left. I_{2}\left(\frac{p_{1}}{p_{1}+p_{2}}, \frac{p_{2}}{p_{1}+p_{2}}\right)\right|\leq \varepsilon_{n-1}
\end{array}
\end{equation}
holds for all $n\geq 3$ and $\left(p_{1}, \ldots, p_{n}\right)\in\Gamma_{n}$, and
\[
\left|I_{3}\left(p_{1}, p_{2}, p_{3}\right)-
I_{3}\left(p_{1}, p_{3, p_{2}}\right)\right|\leq \varepsilon_{1}
\]
holds on $\Gamma_{3}$.
Then there exist $c, d\in\R$ and a
$q*\in [0,1]^{k}$ such that
\begin{equation}\label{stab.syst}
\begin{array}{l}
\left|
I_{n}\left(p_{1}, \ldots, p_{n}\right)-
\left[c\left(\sum^{n}_{i=1}M(p_{i})-1\right)-
d\left(M\left(p_{1}\right)-1\right)\right]\right| \\
\leq
\sum^{n-1}_{k=2}\varepsilon_{k}+
\left(1+\left(n-2\right)M\left(p_{1}+p_{2}\right)\right)K\left(p_{2}\right)
\end{array}
\end{equation}
holds for all $n\geq 2$ and $\left(p_{1}, \ldots, p_{n}\right)\in\Gamma_{n}$, where
the convention $\sum^{1}_{k=2}\varepsilon_{k}=0$ is adapted and the function $K$
is defined by formula (\ref{def.K}).
\end{Thm}
\begin{proof}
The proof runs by induction on $n$.
Let $\left(x, y\right)\in D_{k}$, $n=3$ and substitute
\[
p_{1}=1-x-y
\quad
p_{2}=y
\quad
p_{3}=x
\]
into (\ref{ineq.syst}). Then
\[
\left|
I_{3}\left(\mathbf{1}-x-y, y, x\right)-I_{2}\left(\mathbf{1}-x, x\right)-
M\left(\mathbf{1}-x\right)
I_{2}\left(\mathbf{1}-\frac{y}{\mathbf{1}-x}, \frac{y}{\mathbf{1}-x}\right)
\right|
\leq \varepsilon_{2}
\]
holds.
Hence we get that the function $f:[0,1]^{k}\rightarrow\R$ defined by
\[
f(x)=I_{2}\left(\mathbf{1}-x, x\right) \qquad \left(x\in [0,1]^{k}\right)
\]
satisfies
\begin{equation}
\begin{array}{l}
\left|
f\left(x\right)+M\left(\mathbf{1}-x\right)f\left(\frac{y}{\mathbf{1}-x}\right)-
f(x)-M\left(\mathbf{1}-x\right)f\left(\frac{x}{\mathbf{1}-y}\right)
\right| \\
\leq
\left|
f(x)+M\left(\mathbf{1}-x\right)f\left(\frac{y}{\mathbf{1}-x}\right)-
I_{3}\left(\mathbf{1}-x-y, y, x\right)
\right| \\
+ \left|
I_{3}\left(\mathbf{1}-x-y, y, x\right)-I_{3}\left(\mathbf{1}-y-x, x, y\right)
\right| \\
+
\left |
I_{3}\left(\mathbf{1}-y-x, y, x\right)-f(y)-M\left(\mathbf{1}-y\right)f\left(\frac{y}{\mathbf{1}-y}\right)
\right|
\leq \varepsilon_{1}+2\varepsilon_{2}
\end{array}
\end{equation}
for all $(x, y)\in D_{k}$.
Thus, by Theorem \ref{main}. we get that there exist $a, b\in\R$ and
a $q*\in [0,1]$ such that
\begin{equation}
\begin{array}{l}
\left|f\left(x\right)-\left[aM\left(x\right)+b\left(M\left(\mathbf{1}-x\right)-1\right)\right]\right|\\
 \leq \left|M\left(q*\right)+M\left(\mathbf{1}-q*\right)-1\right|^{-1} \cdot \\
\left(4\left(\varepsilon_{1}+2\varepsilon_{2}\right) +
3\left(\varepsilon_{1}+2\varepsilon_{2}\right) M\left(\mathbf{1}-xq*\right)\right)
\end{array}
\end{equation}
holds for all $x\in [0,1]^{k}$. Let now
$\left(p_{1}, p_{2}\right)\in \Gamma_{2}$, then we obtain that
\begin{equation}
\begin{array}{l}
\left|
I_{2}\left(p_{1}, p_{2}\right)-
\left[aM\left(p_{2}\right)+b\left(M\left(p_{1}\right)-1\right)\right]\right. \\
\left. \leq
\left|M\left(q*\right)+M\left(\mathbf{1}-q*\right)-1\right|^{-1} \cdot
\left(4\varepsilon +3\varepsilon M\left(\mathbf{1}-p_{2}q*\right)\right)
\right|
\end{array}
\end{equation}
holds. Define $c=a$ and $d=b-a$, then we get that
\[
\begin{array}{l}
\left|
I_{2}\left(p_{1}, p_{2}\right)-
c\left[\sum^{2}_{k=1}M\left(p_{k}\right)-d\left(M(p_{1})-1\right)\right]\right. \\
\left. \leq
\left|M\left(q*\right)+M\left(\mathbf{1}-q*\right)-1\right|^{-1} \cdot
\left(4\varepsilon +3\varepsilon M\left(\mathbf{1}-p_{2}q*\right)\right)
\right|\\
= \sum^{2-1}_{k=2}\varepsilon_{k}+
\left(1-\left(1-1\right)M\left(p_{1}+p_{2}\right)\right)\cdot K\left(p_{2}\right),
\end{array}
\]
hence the statement holds for $n=2$.  Assume now that (\ref{ineq.syst}) holds and
introduce the following notation
\[
J_{n}\left(p_{1}, \ldots, p_{n}\right)=c\left(\sum^{n}_{k=1}M\left(p_{k}\right)-1\right)
+d\left(M\left(p_{1}\right)-1\right)
\]
for all $n\geq 2$, $\left(p_{1}, \ldots, p_{n}\right)\in\Gamma_{n}$.
It can be easily seen that $J_{n}:\Gamma_{n}\rightarrow\R$ is an
$M$--recursive and $3$-semisymmetric information measure $\left(n\in\N\right)$.
Therefore we get that
\begin{equation}
\begin{array}{l}
\left|I_{n+1}\left(p_{1}, \ldots, p_{n+1}\right)-J_{n+1}\left(p_{1}, \ldots, p_{n+1}\right) \right| \\
=
\left|I_{n+1}\left(p_{1}, \ldots, p_{n+1}\right)-J_{n}\left(p_{1}+p_{2}, \ldots, p_{n+1}\right)-
M\left(p_{1}+p_{2}\right)I_{2}\left(\frac{p_{1}}{p_{1}+p_{2}}, \frac{p_{2}}{p_{1}+p_{2}}\right) \right|\\
\leq
\left|I_{n+1}\left(p_{1}, \ldots, p_{n+1}\right)-
I_{n}\left(p_{1}+p_{2}, \ldots, p_{n+1}\right)-
M\left(p_{1}+p_{2}\right)I_{2}\left(\frac{p_{1}}{p_{1}+p_{2}}, \frac{p_{2}}{p_{1}+p_{2}}\right)\right| \\
+
\left|
I_{n}\left(p_{1}+p_{2}, \ldots, p_{n+1}\right)- J_{n}\left(p_{1}+p_{2}, \ldots, p_{n}\right)
\right| \\
+
\left|
M\left(p_{1}+p_{2}\right)I_{2}\left(\frac{p_{1}}{p_{1}+p_{2}}, \frac{p_{2}}{p_{1}+p_{2}}\right)-
M\left(p_{1}+p_{2}\right)J_{2}\left(\frac{p_{1}}{p_{1}+p_{2}}, \frac{p_{2}}{p_{1}+p_{2}}\right)
\right| \\
\leq
\varepsilon_{n}+ \sum^{n-1}_{k=2}\varepsilon_{n}+
\left(1+\left(n-2\right)M\left(p_{1}+p_{2}\right)\right)K\left(p_{2}\right)+
M\left(p_{1}+p_{2}\right)K\left(p_{2}\right)\\
=
\sum^{n}_{k=2}\varepsilon_{k}+\left(1+\left(n-1\right)M\left(p_{1}+p_{2}\right)\right)K\left(p_{2}\right),
\end{array}
\end{equation}
 for all $\left(p_{1}, \ldots, p_{n}\right)\in \Gamma_{n+1}$, that is, (\ref{stab.syst}) holds for
 $n+1$ instead of $n$, which ends the proof.
\end{proof}

\begin{Rem}
Our argument does not work in case $M$ is a projection, i.e., we cannot prove stability concerning the
fundamental equation of information in this case neither on the closed nor on the open domain.
\end{Rem}

\end{document}